\documentclass[12pt]{article}
\usepackage{subfigure}
\usepackage[section]{placeins}
\usepackage{amsmath,amsthm,latexsym,amsfonts,amssymb,mathrsfs}
\usepackage{amsbsy}
\usepackage{url}
\usepackage[normalem]{ulem} 
\newcommand{\stkout}[1]{\ifmmode\text{\sout{\ensuremafth{#1}}}\else\sout{#1}\fi}

\usepackage[usenames]{color}

\newcommand{\W}{{\mathcal W}}

\usepackage{tikz}

\def\N{\mathbb{N}}

\newcommand{\calL}{\mathcal{L}}

\newcommand{\mD}{{\boldsymbol{D}}}

\newcommand{\mM}{{\boldsymbol{M}}}
\newcommand{\mS}{{\boldsymbol{S}}^n}

\newcommand{\I}{\mathcal{I}}

\newcommand{\vUh}{{\rm {\boldsymbol {\rm U}}}_h}
\newcommand{\vFh}{{\rm {\boldsymbol {\rm F}}}_h}
\newcommand{\vWh}{{\rm {\boldsymbol {\rm V}}}_h}
\newcommand{\vu}{u}

\newcommand{\bsx}{{\boldsymbol{x}}}
\newcommand{\bsy}{{\boldsymbol{y}}}

\newcommand{\bg}{{\bf g}}

\newcommand{\bbeta}{\boldsymbol \beta}

\newcommand{\R}{\mathbb{R}}

\newcommand{\calS}{\mathcal{S}}

\numberwithin{equation}{section}

\newcommand{\iprod}[1]{\langle#1 \rangle}

\newcommand{\argmin}{\arg\min}
\newcommand{\ud}{\mathrm{d}}
\newtheorem{theorem}{Theorem}[section]

\newtheorem{lemma}{Lemma}[section]

\title{Time-fractional diffusion equations with randomness, and efficient numerical estimations of expected values\thanks{This work was supported by the Australian Research Council grant DP220101811.}\thanks{School of Mathematics and Statistics, University of New South Wales, Sydney, Australia}}
\author{Josef Dick, Hecong Gao, William McLean and Kassem Mustapha}
\begin{document}

\maketitle
\begin{abstract}
In this work, we explore a time-fractional diffusion equation of order $\alpha \in (0,1)$ with a stochastic diffusivity parameter. We focus on efficient estimation of the expected values (considered as an infinite dimensional integral on the parametric space corresponding to the random coefficients) of linear functionals acting on the solution of our model problem. To estimate the expected value computationally, the infinite expansions of the random parameter need to be truncated. Then we approximate the high-dimensional integral over the random field using a high-order quasi-Monte Carlo  method. This follows by approximating the deterministic solution over the space-time domain via a second-order accurate time-stepping scheme in combination with a spatial discretization by Galerkin finite elements. Under reasonable regularity assumptions on the given data, we show some regularity properties of the continuous solution and investigate the errors from estimating the expected value. We report on numerical experiments that complement the theoretical results.
\end{abstract}
\section{Introduction}
In this work, we are interested in estimating the expected value of the random solution of the following stochastic time-fractional diffusion equation  \cite{Sposinietal2020,Wangetal2020}:
\begin{equation}
\label{eq:fractional equation} 
\partial_t^{\alpha} u(\bsx,t,\bsy)+{\mathcal A} u(\bsx,t,\bsy)=f(\bsx,t) \quad
\text{for $(\bsx,t)\in \Omega\times (0,T]$,}
\end{equation}
subject to a homogeneous Dirichlet boundary condition, $u(\bsx,t)= 0$ for $(\bsx,t)\in\partial\Omega\times (0,T]$, and the initial condition $u(\bsx,0,\bsy)=g(\bsx)$ for $\bsx \in\Omega$. Here $T>0,$  $\Omega\subset \mathbb{R}^d$ (with $d=1,2,3)$ is a convex polyhedral domain and  $0<\alpha<1$.  The fractional-order time derivative operator~$\partial_t^\alpha$ is of Caputo type, that is,
\[ \partial_t^{\alpha} v(t):=\I^{1-\alpha}v'(t)
=\int_0^t\omega_{1-\alpha}(t-s)v'(s)\,\ud s\quad\text{with}
\quad \omega_{1-\alpha}(t):=\frac{t^{-\alpha}}{\Gamma(1-\alpha)},\]
where  $\I^{1-\alpha}$ is the Riemann--Liouville fractional integral operator with respect to~$t$ and where $v'=\partial v/\partial t$. ($\Gamma$ denotes the usual gamma function.)  In \eqref{eq:fractional equation}, ${\mathcal A}$ is the spatial elliptic stochastic operator defined by 
\[{\mathcal A}w(\bsx,\bsy)=-\nabla\cdot(\kappa(\bsx,\bsy)\nabla w(\bsx,\bsy))\,,\]
where the uncertainty in the diffusivity~$\kappa$ is expressed in the expansion 
\begin{equation}\label{KLexpansion}
\kappa(\bsx, \bsy) = \kappa_0(\bsx) +\sum_{j=1}^\infty y_j \psi_j(\bsx),
\end{equation}
with $\bsy= (y_j)_{j\ge1}\in \mathfrak{D}:=(-\frac{1}{2},\frac{1}{2})^\N$ consisting of a countable number of parameters $y_j$,  which are assumed to be i.i.d. uniformly distributed.  The functions in the sequence $\{\psi_j\}$ are assumed to be independent in the space $L^2(\Omega)$ (if the functions $\psi_j$ are derived from a Karhunen-Lo\'eve expansion than they are even orthogonal). The function  $\kappa_0$ is the mean of the random field $\kappa$ with respect to the random variables $\bsy$.

To ensure that  $\kappa$ is well-defined for all parameter vectors $\bsy \in \mathfrak{D}$, we impose the conditions 
\begin{equation}\label{ass A1}
 \kappa_0   \in L^\infty(\Omega) \quad {\rm and}\quad  \sum_{j=1}^\infty \|\psi_j\|_{L^\infty(\Omega)} < \infty,
\end{equation}
and to ensure the well-posedness of problem \eqref{eq:fractional equation}, we assume there exist constants $\kappa_{\min}$~and $\kappa_{\max}$ such that
\begin{equation}\label{eq: kappa min max}
0  < \kappa_{\min} \le \kappa \le  \kappa_{\max}<\infty
\quad\text{on $\Omega \times \mathfrak{D}$.}
\end{equation}

For a deterministic diffusivity $\kappa=\kappa(\bsx)$, well-posedness, regularity properties, and numerical solutions for subdiffusion problems  of the form \eqref{eq:fractional equation} were studied very extensively over the last fifteen years \cite{JinZhou2023Book,Mustapha2018,Mustapha2020}. Furthermore, the inverse problem of recovering/identifying the diffusion coefficient $\kappa$  has attracted much interest \cite{ChengNakagawaYamamotoYamazaki2009,
JinZhou2023,LiGuJia2012,LiYamamoto2019,WeiLi2018,Zhang2016}. However, and to the best of our knowledge, the case of a random diffusivity coefficient $\kappa=\kappa(\bsx,\bsy)$ has not been studied so far,  despite its practical relevance.

For a continuous linear functional $\calL:L^2(\Omega)\to \R$ and for each fixed time $t \in (0,T],$  we are interested in efficient  approximation of the  expected value of $\calL(u(t))$, that is,
\begin{equation}\label{F} 
E\bigl(\calL(u(t))\bigr):=\int_{\mathfrak{D}}\calL\big( u(\cdot,t,\bsy)\big)\,\ud\bsy.
\end{equation}
Here, $\ud\bsy$ is the uniform probability measure on~$\mathfrak{D}$, and $u(\cdot,t,\bsy)$ is the random solution of \eqref{eq:fractional equation} at time $t$. 

To estimate $E(\calL(u(t))),$  we have to deal with three different sources of errors: a dimension truncation error from truncating the  infinite expansions in\eqref{KLexpansion}; a sampling error from approximating the expected value; and the error from approximating the continuous solution over the space-time domain $\Omega\times (0,T)$. Here, we rely on the Galerkin finite element method (FEM) for the spatial discretization, combined with a second-order accurate time-stepping method (over graded meshes) from our recent paper \cite{MMD2024}. Noting that, due to the presence of the non-local Caputo derivative 
in~\eqref{eq:fractional equation}, the solution $u$ suffers from a weak singularity at $t=0$ even when the given data are smooth. This has a direct impact on the accuracy, and consequently the convergence rates, of the numerical methods in both variables, time and space. To overcome this delicate issue, different approaches were applied, including corrections, graded meshes, and convolution quadrature. 

The second source of error is from approximating the integral over the parameter space $\mathfrak{D}$. Here, we rely on a  high-order  quasi-Monte Carlo (QMC) method \cite{Dick2008,DickKuoGiaNuynsSchwab2014}. In Section \ref{VFR}, we recall some known regularity properties of the continuous solution $u$ of \eqref{eq:fractional equation}, where we require $\kappa(\cdot,\bsy)\in W^{1,\infty}(\Omega)$ for  every $\bsy\in\mathfrak{D}$. To have this, we additionally assume that \begin{equation}\label{ass A4}
\kappa_0  \in W^{1,\infty}(\Omega)\quad\text{and}\quad
\sum_{j=1}^\infty \|\nabla \psi_j\|_{L^\infty(\Omega)}<\infty.
\end{equation} We also show in Subsection \ref{VFR xt} some other regularity estimates of $u$ with respect to $\bsx$ and $t$.  In addition, we show in Subsection \ref{VFR y} some regularity properties of $u$  with respect to the random vector $\bsy$. These properties are required for our error analysis.

Since we cannot sample directly from the infinite sum in \eqref{KLexpansion}, we consider the truncation $\widehat \kappa(\bsx, \bsy)=\kappa(\bsx, \widehat \bsy)$, where $\widehat \bsy=(y_1,\ldots,y_z,0,\ldots)$ is the truncated vector of~$\bsy$ for some~$z\ge 1$,  obtained from $\bsy$ by setting $y_j=0$ for $j>z$. The error from truncating the infinite series expansion  in \eqref{KLexpansion} is investigated in Theorem \ref{Truncating error}. To minimize  the errors, the sequence of functions~$\{\psi_j\}_{j\ge 1}$ is ordered so that
\begin{equation}\label{ass A5}
\| \psi_j \|_{L^\infty(\Omega)} \ge
\| \psi_{j+1} \|_{L^\infty(\Omega)},\quad\text{for $j\ge 1$,}
\end{equation}
or in other words, so that the sequence  $\{\|\psi_j\|_{L^\infty(\Omega)}\}_{j\ge 1}$ is  nonincreasing. We also require that 
\begin{equation}\label{ass A3}
\sum_{j=1}^\infty \|\psi_j\|^p_{L^\infty(\Omega)} < \infty,\quad
\text{for some~$p$ satisfying $0<p<1$.}
\end{equation}
Then, an application of the Stechkin inequality  
\[
\sum_{j \ge s+1} \delta_j\le C_\varsigma\,
s^{1-\frac{1}{\varsigma}}\Big(\sum_{j \ge 1} \delta_j^\varsigma\Big)^{\frac{1}{\varsigma}},\quad{\rm for}~~0< \varsigma < 1,
\]
where $\{\delta_j\}_{j\ge1}$ is a nonincreasing sequence of positive numbers, leads to 
\begin{equation}\label{error in kappa}
|\kappa(\bsx,\bsy) - \widehat \kappa(\bsx, \bsy)| 
\le \frac{1}{2} \sum_{j=z+1}^\infty \|\psi_j\|_{L^\infty(\Omega)} \le C_p\, z^{1-1/p},
\quad\text{for $(\bsx,\bsy)\in\Omega\times\mathfrak{D}$.}
\end{equation}

In Section~\ref{sec: FEM} we approximate the solution $u(\cdot,\cdot,\bsy)$ of \eqref{eq:fractional equation} over $\Omega\times (0,T)$ by using a Galerkin FEM in space, combined with a second-order time-stepping scheme. To improve the accuracy, the time mesh is graded near the origin to compensate for the singular behavior of the continuous solution. The matrix implementation of the numerical scheme is discussed briefly. Section \ref{Sec: main results} introduces a high-order QMC rule and establishes our main results.  To support our theoretical finding, we present some numerical results in Section \ref{Sec: Numeric}.
\section{Regularity properties}\label{VFR}
This section is dedicated to the regularity properties of the continuous solution $u$ of \eqref{eq:fractional equation} with respect to the spatial  variable~$\bsx$, the time variable $t$, and the random variable $\bsy$. 
We will use the following spaces and notations.  The inner product and norm in $L^2(\Omega)$ are written as $\iprod{\cdot,\cdot}$~and $\|\cdot\|$, respectively. For $\ell \ge 1,$ the norm on the Sobolev space $H^\ell(\Omega)$ is denoted by $\|\cdot\|_\ell$, and we let
\[
V:=H^1_0(\Omega)=\{ w \in H^1(\Omega): w|_{\partial \Omega} = 0 \}. 
\]
For the half-open time interval~$J=(0,T]$, we will write the norm 
in~$L^2(J,L^2(\Omega))$ as
\[
\|f\|_{L^2(J,\Omega)}=\biggl(\int_J\|f(\cdot,t)\|^2\,\ud t\biggr)^{1/2}.
\]
Our analyses in this and the forthcoming sections use the following technical lemmas. 

The proof of the next lemma follows from a result of Mustapha and Sch\"otzau~\cite[Lemma 3.1~(iii)]{MustaphaSchoetzau2014} and the inequality $\cos((1-\alpha)\pi/2)\geq \alpha$ for $0<\alpha<1.$

\begin{lemma}\label{lem: continuity}
For $0<\alpha<1$ and for $\epsilon >0,$ we have 
\[
\int_0^t\iprod{\I^{1-\alpha}\phi,\psi}\,\ud s
\leq \epsilon \int_0^t\iprod{\I^{1-\alpha}\phi,\phi}\,\ud s 
+\frac{1}{4\epsilon \alpha^2}\int_0^t\iprod{\I^{1-\alpha}\psi,\psi}\,\ud s,
\quad\text{for $\epsilon>0$.}\]
\end{lemma}

The next estimate was proved by McLean et al.~\cite[Lemma 2.3]{McLeanEtAl2019}.
\begin{lemma}\label{lem: tech 2}
For $0<\alpha<1$, if the function $\phi: [0,T] \to L_2(\Omega)$ is continuous with $\phi(0)=0$, and if its restriction to $(0,T]$ is piecewise diﬀerentiable with $\phi(t)\le C t^{-\mu}$ for $0<t\le T$ and some nonnegative 
constant~$\mu <1-\alpha$, then
\[
\|\phi(t)\|^2\leq 2\omega_{1+\alpha}(t)\int_0^t\iprod{\I^{1-\alpha}\phi',\phi'}\,\ud s,
\quad\text{for $t>0$.}
\]
\end{lemma}

\subsection{Regularity with respect to ${\bsx}$ and $t$}\label{VFR xt}
 In this subsection, we state some known regularity properties of $u$ with respect to the variables $\bsx$ (space)~and $t$ (time). These properties are required to guarantee the convergence of our numerical solution on $\Omega$ for different time levels and for every $\bsy \in \mathfrak{D}$. Moreover, we show some other regularity properties that are needed for showing regularity with respect to the random variable~$\bsy$ (see 
 Theorem~\ref{lem: vu bound y}) and which are also needed for estimating the error from truncating the infinite series in~\eqref{KLexpansion} in the forthcoming section.     

The results of Jin et al.~\cite[Theorems 1 and 2]{JinLiZhou2020} and of Sakamoto and Yamamoto~\cite[Theorems 2.1 and 2.2]{SakamotoYamamoto2011} imply that for 
every $\bsy \in \mathfrak{D}$, and for $0\le \ell \le 1$,   
\begin{equation}\label{Reg x and t}
\|\nabla u(t,\bsy)\|+\|\nabla \partial_t {\mathfrak U}(t,\bsy)\|\\
\le Ct^{-\alpha/2} {\mathcal R}^\ell(t),\quad\text{for $t>0$,}
\end{equation}
provided that the assumptions \eqref{ass A1} and \eqref{ass A4} are satisfied. Here,
\begin{equation}\label{eq: frak U cal R}
\mathfrak{U}(t,\bsy)=tu(t,\bsy)~\text{and}~
\mathcal{R}^\ell(t)=t^{\alpha\ell/2}\|g\|_{\dot H^\ell(\Omega)}+t^{\alpha}\|f(0)\|
+t^{1+\alpha/2}\I^{\alpha/2}(\|f'\|)(t).
\end{equation}
In the next lemma, we prove different technical inequalities that will be used to show a required regularity property of $u$ with respect to the random variable $\bsy$ in the forthcoming subsection. For the error analysis from the time discretization, for every  $\bsy \in \mathfrak{D}$,  we assume that 
\begin{equation}\label{time regularity}
t\|u'(t,\bsy)\|+t^2\|u''(t,\bsy)\|+t^3\|u'''(t,\bsy)\|\le C t^\sigma,\quad
\text{for $t>0$,}
\end{equation}
and for some positive~$\sigma$. In addition, for the error analysis from the Galerkin FEM, we impose that for $t>0,$ 
\begin{equation}\label{spatial regularity}
\|u'(t,\bsy)\|_2+ t\|u''(t,\bsy)\|_2\le Ct^{\nu-1},
\end{equation}
for some constant~$\nu > 0$. For instance, if $f\equiv 0$ and $g \in \dot H^b(\Omega)$ for some $0\le b \le 2$, then \eqref{time regularity} and \eqref{spatial regularity} hold true for $\sigma=b\alpha/2$ and $\nu=(b/2-1)\alpha$ 
\cite[Theorems 4.2 and 4.4]{McLean2010}.  

Consistent with our earlier notation for the fractional integral, we write $$\mathcal{I}v(t)=\mathcal{I}^1v(t)=\int_0^tv(s)\,\ud s.$$
\begin{lemma}\label{lem: reg bound}
For a given space-time dependent $d$-dimensional vector function~$\theta$, let $w$ be the solution of the deterministic  time-fractional diffusion equation 
\begin{equation}\label{eq: 1 truncate n}
\partial_t^\alpha w+\mathcal{A} w= -\nabla\cdot\theta\quad\text{on $\Omega\times(0,T]$,}
\end{equation}
subject to homogeneous Dirichlet boundary conditions, $w=0$ on $\partial\Omega\times (0,T]$, and zero initial condition, $w(\bsx,0)=0$ for $\bsx\in\Omega$. Then, with $\W  (x,t)=t w(x,t)$ and $\Theta(\bsx,t)=t \theta(\bsx,t)$, we have  for $t>0$,    
\begin{align}
\I(\|\sqrt{\kappa}\nabla w\|^2)(t)&\le\I(\|\theta/\sqrt{\kappa}\|^2)(t),
\label{eq: bounda of nabla w} \\
\I \big(\|\sqrt{\kappa} \nabla \W'  \|^2\big)(t)
    &\le 2\I \big(\|\theta/\sqrt{\kappa}\|^2+\|\Theta'/\sqrt{\kappa}\|^2\big)(t),
    \label{eq: bound of nabla v'}\\
\|\sqrt{\kappa} \nabla w(t)\|^2
&\le 2t^{-1}\I\big(\|\theta/\sqrt{\kappa}\|^2+\|\Theta'/\sqrt{\kappa}\|^2\big)(t),
\label{eq: bound of nabla v} \\
\|w(t)\|^2&\le Ct^{\alpha-1}\I\bigl(
    \|\theta/\sqrt{\kappa}\|^2+\|\Theta'/\sqrt{\kappa}\|^2\bigr)(t). \label{eq: desired 4}
\end{align}
\end{lemma}

\begin{proof}
To show the first desired estimate, we take the inner product of \eqref{eq: 1 truncate n} with~$w$, integrate in time, and then apply the divergence theorem and the Cauchy--Schwarz inequality,  to obtain    
\begin{equation}\label{eq:ww}
\I(\iprod{\partial_t^\alpha w,w}+\|\sqrt{\kappa}\nabla w\|^2)(t)
    =\I(\iprod{\theta, \nabla w})(t).    
\end{equation}
Using 
\[2|\iprod{\theta, \nabla w}|
    \le \|\theta/\sqrt{\kappa}\|^2
    +\|\sqrt{\kappa}\nabla w\|^2,
\]
and then, cancelling the common terms, yields   
\begin{equation}\label{eq: bound of I nabla w}
2\I(\iprod{\partial_t^\alpha w,w})(t)+\I(\|\sqrt{\kappa}\nabla w\|^2)(t)
\le \I( \|\theta/\sqrt{\kappa}\|^2)(t).
\end{equation}
Since $w(\bsx,0)=0$, $\I(\iprod{\partial_t^\alpha w,w})(t)=\I(\iprod{\I^{1-\alpha} w',w})(t)=\I(\iprod{\partial_t (\I^{1-\alpha} w),w})(t)\ge 0$, and therefore, the proof of \eqref{eq: bounda of nabla w} is completed.    

For later use, set $\tilde w(t)=t w'(t)$ and notice $\iprod{\mathcal{A}w(t),\tilde w(t)}=(t/2)(d/dt)\|\sqrt{\kappa}\nabla w(t)\|^2$. By taking the inner product of \eqref{eq: 1 truncate n} with $\tilde w$, integrating in time, and then integrating by parts and multiplying through by~$2$, we get
\begin{multline*}
2\I(\iprod{\partial_t^\alpha w,\tilde w})(t)
+t\|\sqrt{\kappa}\nabla w(t)\|^2-\I(\|\sqrt{\kappa}\nabla w\|^2)(t)\\
=2\I(\iprod{\theta, \nabla \tilde w})=2\I(\iprod{\theta, \nabla \W'- \nabla w})(t).
\end{multline*}
Adding this equation to twice~\eqref{eq:ww}, and using $\W'=w+\tilde w$, yields
\begin{equation}\label{eq:wtildew}
\I(\iprod{\partial_t^\alpha w,\W'})(t)\le \I(\iprod{\theta, \nabla \W'})(t).
\end{equation}
To show \eqref{eq: bound of nabla v'},  we start by multiplying both sides 
of~\eqref{eq: 1 truncate n} by~$t$, then using the identity~\cite[(2.5)]{Mustapha2018}
\[t 
\partial_t^\alpha w(t)=\partial_t^\alpha \W  (t)-\alpha\I^{1-\alpha}w(t)
    -t\omega_{1-\alpha}(t)w(0),
\]
and the zero initial condition on~$w$, to obtain
\begin{equation}\label{v eq}
\partial_t^{\alpha} \W   +{\mathcal A} \W  = \alpha\I^{1-\alpha} w-t\nabla\cdot\theta\,.
\end{equation}
Applying $\partial_t$ to both sides of~\eqref{v eq} and using $\W'(0)=w(0)=0$, we get  
\[
\partial_t^\alpha \W'   +{\mathcal A} \W'  = \alpha \partial_t^\alpha w-\nabla\cdot\Theta'.
\]
Taking the  inner product with $\W'$, then using the divergence theorem, applying $\I$ to both sides, and using \eqref{eq:wtildew}, we deduce that 
\begin{multline*}
\I(\iprod{\partial_t^\alpha \W'  ,\W'  })(t)+\I(\|\sqrt{\kappa} \nabla \W'  \|^2)(t)
=\I\big(\alpha\iprod{\partial_t^\alpha w,\W'  }+\iprod{\Theta',\nabla \W'  }\big)(t)\\
\le \I\big(\iprod{\alpha\theta+\Theta',\nabla \W'  }\big)(t).
\end{multline*}
However, by the Cauchy--Schwarz inequality,
\[
|\iprod{\alpha\theta+\Theta',\nabla \W'  }|
\le \|(\alpha\theta+\Theta')/\sqrt{\kappa}\|\,\|\sqrt{\kappa}\nabla \W'\| 
\le \|\theta/\sqrt{\kappa}\|^2+\|\Theta'/\sqrt{\kappa}\|^2
+\tfrac12\|\sqrt{\kappa}\nabla \W'\|^2,
\]
and thus, after cancelling the common terms,  we reach  
\[2\I(\iprod{\partial_t^\alpha \W'  ,\W'  })(t)+\I(\|\sqrt{\kappa}\nabla \W'  \|^2)(t)
\le 2\I(\|\theta/\sqrt{\kappa}\|^2)(t)+2\I(\|\Theta'/\sqrt{\kappa}\|^2)(t).\]
Since $\I(\iprod{\partial_t^\alpha \W'  ,\W'  })(t)\ge 0$, the proof 
of~\eqref{eq: bound of nabla v'} is completed. 

Because  $w(t)=t^{-1}\W(t)$ and
\[
\|\sqrt{\kappa}\nabla\W(t)\|^2=\|\sqrt{\kappa}\nabla(\W(t)-\W(0))\|^2
\le \big(\I(\|\sqrt{\kappa} \nabla \W'\|)(t)\big)^2
\le t\I(\|\sqrt{\kappa}\nabla\W'\|^2)(t),
\]
the estimate in~\eqref{eq: bound of nabla v}  follows immediately from \eqref{eq: bound of nabla v'}.

The main focus is now to prove \eqref{eq: desired 4}.  By Lemma~\ref{lem: tech 2},
\begin{equation}\label{eq: w squared}
\|w(t)\|^2=t^{-2}\|\W(t)\|^2\le\frac{2t^{\alpha-2}}{\Gamma(1+\alpha)}\,
\I(\iprod{\partial_t^\alpha\W,\W'})(t),
\end{equation}
and we proceed to estimate the right-hand side. Taking the  inner product of \eqref{v eq} with $\W'$, then applying the divergence theorem, integrating in time, and using $\W  (0)=0$, we conclude that 
\begin{equation}\label{eq: desired 4 step}
\I(\iprod{\partial_t^\alpha\W,\W'})(t)+\tfrac12\|\sqrt{\kappa}\nabla\W(t)\|^2
=\I\big(\alpha\iprod{\I^{1-\alpha}w,\W'}+\iprod{\Theta,\nabla\W'}\big)(t).
\end{equation}
By applying Lemma~\ref{lem: continuity} with $\epsilon=\frac{1}{2\alpha}$, and noting
that $\I^{1-\alpha}\W'=\partial_t^\alpha\W$, we see that
\[
\alpha\I\bigl(\iprod{\I^{1-\alpha}w,\W'}\bigr)(t)
\le\tfrac12\I\bigl(\iprod{\I^{1-\alpha}w,w}\bigr)(t)
    +\tfrac12\I\bigl(\iprod{\partial_t^\alpha\W,\W'}\bigr)(t),
\]
and using the Cauchy--Schwarz inequality followed by~\eqref{eq: bound of nabla v'}, 
\begin{align*}
\I\bigl(\iprod{\Theta,\nabla\W'}\bigr)(t)
&\le\int_0^t s\|\theta(s)/\sqrt{\kappa}\|\|\sqrt{\kappa}\nabla\W'(s)\|\,\ud s\\
&\le\frac12\int_0^t s\bigl(\|\theta(s)/\sqrt{\kappa}\|^2
    +\|\sqrt{\kappa}\nabla\W'(s)\|^2\bigr)\,\ud s\\
&\le\tfrac32t\I\bigl(\|\theta/\sqrt{\kappa}\|^2\bigr)
    +\tfrac12t\I\bigl(\|\Theta'/\sqrt{\kappa}\|^2\bigr).
\end{align*}
After inserting these two bounds into~\eqref{eq: desired 4 step}, cancelling the 
term~$\tfrac12\I\bigl(\iprod{\partial_t^\alpha\W,\W'}\bigr)(t)$, and multiplying 
through by~$2$, it follows that
\begin{equation}\label{eq: bound of I alpha v n}
\I(\iprod{\partial_t^\alpha\W,\W'})(t)+\|\sqrt{\kappa} \nabla \W  (t)\|^2
\le \I(\iprod{\I^{1-\alpha} w,w})(t)
+t\I\bigl(3\|\theta/\sqrt{\kappa}\|^2+\|\Theta'/\sqrt{\kappa}\|^2\bigr)(t).
\end{equation}
To estimate $\I(\iprod{\I^{1-\alpha} w,w})(t)$, we note that
$\partial_t^\alpha w(t)=\partial_t\I^{1-\alpha}w(t)$ because $w(0)=0$, so after
integrating \eqref{eq: 1 truncate n} in time and then taking the inner product with~$w$,
\[ 
\iprod{\I^{1-\alpha}w,w}+\iprod{\kappa \nabla \I w,\nabla w}
    = \iprod{\I\theta, \nabla w}.
\]
Since $\iprod{\kappa\nabla\I w,\nabla w}=(d/dt)\|\sqrt{\kappa}\nabla\I w\|^2$, by again
integrating in time, and then using $\I w(0)=0$ followed by the Cauchy--Schwarz inequality,
\begin{align*}
\I(\iprod{\I^{1-\alpha} w,w})(t)+\tfrac12\|\sqrt{\kappa} \nabla \I w\|^2
&\le \I\bigl(\|\I\theta/\sqrt{\kappa}\|\|\sqrt{\kappa}\nabla w\|\bigr)(t)\\
&\le\frac12\int_0^t\bigl(s^{-1}\|\I\theta/\sqrt{\kappa}\|^2
    +s\|\sqrt{\kappa}\nabla w\|^2\bigr)\,\ud s.
\end{align*}
Inserting this inequality in~\eqref{eq: bound of I alpha v n} and observing that,
by~\eqref{eq: bound of nabla v}, 
\[
\frac12\int_0^ts\|\sqrt{\kappa}\nabla w\|^2\,\ud s
    \le t\I\bigl(\|\theta/\sqrt{\kappa}\|^2+\|\Theta'/\sqrt{\kappa}\|^2\bigr)(t),
\]
we arrive at
\[
\I(\iprod{\partial_t^\alpha\W,\W'})(t)
\le2t\I\bigl(2\|\theta/\sqrt{\kappa}\|^2+\|\Theta'/\sqrt{\kappa}\|^2\bigr)(t)
    +\frac12\int_0^ts^{-1}\|\I\theta/\sqrt{\kappa}\|^2\,\ud s.
\]
Furthermore, $\|\I\theta(s)/\sqrt{\kappa}\|^2
\le(\int_0^s\|\theta(q)/\sqrt{\kappa}\|\,\ud q)^2
\le s\int_0^s\|\theta(q)/\sqrt{\kappa}\|^2\,\ud q$, so
\[
\frac12\int_0^ts^{-1}\|\I\theta/\sqrt{\kappa}\|^2\,\ud s
\le\frac{1}{2}\int_0^t\int_0^s\|\theta(q)/\sqrt{\kappa}\|^2\,\ud q\,\ud s
\le\tfrac{t}2 \I(\|\theta/\sqrt{\kappa}\|^2)(t).
\]
The estimate~\eqref{eq: desired 4} now follows from~\eqref{eq: w squared}.
\end{proof}

\subsection{Regularity with respect to  ${\bsy}$}\label{VFR y}
In this subsection, we investigate the regularity properties of the parametric solution $u$ of \eqref{eq:fractional equation} with respect to the random variable $\bsy$.  These properties are needed to guarantee the convergence of the QMC method. For convenience,  we introduce the following notations: let  ${\calS}$ be the set of infinite  vectors~$\bbeta=(\beta_j)_{j\ge 1}$ with nonnegative integer entries such that   $|\bbeta|:=\sum_{j\ge 1} \beta_j<\infty$. That is, sequences of nonnegative integers for which only finitely many entries are nonzero. 
For~$\bbeta=(\beta_j)_{j\ge 1}\in\calS$, define the mixed partial derivative operator
\[
\partial_{\bsy}^{\bbeta}
:= \frac{\partial^{|\bbeta|}}{\partial_{y_1}^{\beta_1}\partial_{y_2}^{\beta_2}\cdots}\,. 
\]
Recall the notation \eqref{eq: frak U cal R}, and for convenience, for $0\le \ell\le 1,$ we set 
\begin{equation}\label{eq: E(t)}
\mathcal{E}^\ell(t)=\Big(t^{\alpha-1}\int_0^t s^{-\alpha} (\mathcal{R}^\ell(s))^2\,ds\Big)^{1/2}.
\end{equation}
\begin{theorem}\label{lem: vu bound y}
Assume that \eqref{ass A1} and \eqref{ass A5}  are satisfied. Then, for every $\bsy\in \mathfrak{D}$ and $\bbeta\in\calS$, the parametric solution $\vu(\cdot,\cdot,\bsy)$ of the problem \eqref{eq:fractional equation} satisfies
\begin{equation}\label{mixed estimate}
\I({\mathcal G}^{\bbeta})(t) 
\le (|\bbeta|!)^2 {\bf b}^{2\bbeta} \I({\mathcal G}^{\bf 0})(t),
\quad\text{for $t>0$,}
\end{equation} 
where 
\[
\mathcal{G}^{\bbeta}(t)=\|\sqrt{\kappa}\nabla(\partial_{\bsy}^{\bbeta} u(t,\bsy))\|^2
    +\|\sqrt{\kappa}\nabla(\partial_{\bsy}^{\bbeta} {\mathfrak U}'(t,\bsy))\|^2
\]
with
\[
{\bf b}^{\bbeta}=\prod_{i\ge 1}  b_i^{\beta_i}\quad\text{and}\quad
b_j=\frac{\sqrt{2}}{\kappa_{\min}}\,\|\psi_j\|_{L^\infty(\Omega)}.
\]
Furthermore, for $t>0$, by applying the regularity properties \eqref{eq: bound of nabla v} and \eqref{eq: desired 4}, and also the regularity property in \eqref{Reg x and t}, we obtain the following  pointwise-in-time regularity estimate: for $0<t<T$ and for $0\le \ell \le 1$,
 \begin{multline}\label{mixed estimate 2}
\|\partial_{\bsy}^{\bbeta} u(t,\bsy)\|
+t^{\alpha/2}\|\sqrt{\kappa}\nabla(\partial_{\bsy}^{\bbeta} u(t,\bsy))\|
    \le  t^{(\alpha-1)/2} |\bbeta|!
{\bf b}^{\bbeta} \bigl(\I({\mathcal G}^{\bf 0})(t)\bigr)^{1/2}\\
=t^{(\alpha-1)/2} |\bbeta|! {\bf b}^{\bbeta} \biggl(\int_0^t\bigl[
\|\sqrt{\kappa}\nabla u(s,\bsy)\|^2+\|\sqrt{\kappa}\,\nabla\mathfrak{U}'(s,\bsy)\|^2
    \bigr]\,\ud s\biggr)^{1/2} \\
\le C\,t^{(\alpha-1)/2} |\bbeta|! {\bf b}^{\bbeta} \biggl(
\int_0^t s^{-\alpha}(\mathcal{R}^\ell(s))^2\,\ud s\biggr)^{1/2}=C\,|\bbeta|! {\bf b}^{\bbeta} {\color{blue}\mathcal{E}^\ell(t)}.
\end{multline}
\end{theorem}
\begin{proof}
Differentiating both sides of~\eqref{eq:fractional equation} with respect to the  variable~$y_j$, we find the following recurrence after a tedious calculation 
\begin{equation}\label{recurrence}
 \partial_t^{\alpha} (\partial_{\bsy}^{\bbeta}u)+{\mathcal A} (\partial_{\bsy}^{\bbeta} u)
=\sum_{\bbeta}\beta_j \nabla\cdot(\psi_j \nabla(\partial_{\bsy}^{\bbeta-{\bf e}_j} u)),
\end{equation}
where $\sum_{\bbeta}:=\sum_{j,\beta_j\ne 0}$, that is, the sum over the nonzero indices of $\bbeta$, and ${\bf e}_j \in \calS$ denotes the vector with entry $1$ in position $j$ and zeros elsewhere.  From the given boundary and initial condition, we deduce that $(\partial_{\bsy}^{\bbeta}u)(\bsx,t,\bsy)=0$ for $(\bsx,t,\bsy)\in \partial\Omega\times(0,T]\times \mathfrak{D}$ and that $(\partial_{\bsy}^{\bbeta}u)(\bsx,0,\bsy)=0$ for~$(\bsx,\bsy) \in \Omega\times\mathfrak{D}$.  Therefore, an application of the estimate \eqref{eq: bound of nabla v'} 
with $\theta=-\sum_{\bbeta}\beta_j \psi_j \nabla(\partial_{\bsy}^{\bbeta-{\bf e}_j} u)$  yields  
\begin{align*}
\I({\mathcal G}^{\bbeta})(t) &\le 2\int_0^t \Bigl(
\Bigl\|\sum_{\bbeta}\beta_j\psi_j\nabla(\partial_{\bsy}^{\bbeta-{\bf e}_j} u)\Bigr\|^2
    +\Bigl\|\sum_{\bbeta}\beta_j\psi_j
    \nabla(\partial_{\bsy}^{\bbeta-{\bf e}_j}\mathfrak{U}')\Bigr\|^2\Bigr)\,\ud s\\
&\le \int_0^t \Bigl(\sum_{\bbeta}\beta_j b_j \|\sqrt{\kappa}\nabla(\partial_{\bsy}^{\bbeta-{\bf e}_j} u)\|
    \Bigr)^2\,\ud s+\int_0^t\Bigl(
\sum_{\bbeta}\beta_jb_j\|\sqrt{\kappa}\nabla(\partial_{\bsy}^{\bbeta-{\bf e}_j} \mathfrak{U}')\|\Bigr)^2\,\ud s\\
&\le \sum_{\bbeta}\beta_j\sum_{\bbeta}\beta_j b_j^2 \int_0^t \Bigl(\|\sqrt{\kappa}\nabla(\partial_{\bsy}^{\bbeta-{\bf e}_j} u)\|^2+\|\sqrt{\kappa}\nabla(\partial_{\bsy}^{\bbeta-{\bf e}_j} {\mathfrak U}')\|^2
    \Bigr)\,\ud s,
\end{align*}
and consequently, 
\begin{equation}\label{estimate 2}
\I({\mathcal G}^{\bbeta})(t)
\le |\bbeta|\sum_{\bbeta}\beta_j b_j^2 \I({\mathcal G}^{\bbeta-{\bf e}_j})(t)\,.
\end{equation}
We complete our proof by induction on~$|\bbeta|$. For $|\bbeta|=1$, without loss of generality, we assume that $\beta_1=1$ and all the other entries in $\bbeta$ are zeros. Then,   from the above contribution, it is 
clear that $\I({\mathcal G}^{\bbeta})\le  b_1^2 \I(\mathcal{G}^{\bf 0})$, and so  \eqref{mixed estimate} holds when~$|\bbeta|=1$. Now, assuming that \eqref{mixed estimate} is true for $|\bbeta|=n$, our task is to prove the result for $|\bbeta|=n+1$. In this case, we have at most $n+1$ non-zero entries in $\bbeta.$ Again, and without loss of generality, we assume that $\beta_1,\beta_2,\cdots,\beta_{n+1}$ are the only possible non-zero entries in $\bbeta.$  From \eqref{estimate 2} and the induction hypothesis, we have  
\begin{align*}
\I({\mathcal G}^{\bbeta})(t)
    &\le |\bbeta|\sum_{\bbeta} \beta_jb_j^2 \I({\mathcal G}^{\bbeta-{\bf e}_j})(t) \\
    &\le (n+1) (n!)^2 \sum_{j=1}^{n+1} 
    \beta_jb_j^2 {\bf b}^{2(\bbeta-{\bf e}_j)}\I(\mathcal{G}^{\bf 0})(t)
    =(n+1)^2 (n!)^2 {\bf b}^{2\bbeta}\I(\mathcal{G}^{\bf 0})(t),
\end{align*}
and the inductive step is completed. 
\end{proof}

\section{Numerical discretizations over $\Omega\times (0,T]$}\label{sec: FEM}
In this section, we discretize the parametric model 
problem~\eqref{eq:fractional equation} over $\Omega\times (0,T)$ using the numerical method from a recent paper \cite{MMD2024}. 
This scheme was developed for problems of the form~\eqref{eq:fractional equation} but with a deterministic diffusivity field~$\kappa$. The approximate solution is a continuous piecewise linear polynomial in both variables $t$~and $\bsx$. We integrate locally in time, approximate the continuous solution by a continuous piecewise linear polynomial in time, then define the weak formulation and apply the standard continuous Galerkin procedure in space, thereby defining a fully discrete scheme that achieves second-order accuracy in both time and space.  As~$\alpha\to 1$, the problem \eqref{eq:fractional equation} reduces to the well-known  classical  diffusion equation with random diffusivity and external source, and our numerical scheme in~\eqref{CNFEM} reduces to the well-known Crank--Nicolson Galerkin FEM.  To define our fully-discrete  approximate solution~$u_h(\cdot,\cdot,\bsy)\approx u(\cdot,\cdot,\bsy)$ for every $\bsy\in\mathfrak{D}$, we use the following time mesh: for $\mathcal{N}_t\ge 1$ and with~$\gamma \ge 1$, 
\begin{equation} \label{eq: time mesh}  t_n=(n\,\tau)^\gamma,\quad\text{with $\tau=T^{1/\gamma}/{\mathcal N}_t$ and $\gamma \ge 1$,}\quad\text{for $0\le n\le\mathcal{N}_t$.}
\end{equation}
This time-graded mesh compensates for the singular behaviour of the continuous solution near the origin.  Now, we introduce a family of regular (conforming) triangulations~$\mathcal{T}_h$ of the domain~$\overline{\Omega}$, and let 
$h=\max_{K\in \mathcal{T}_h}h_K$, where $h_K$ denotes the diameter of the element~$K$. Let $V_h\subset H^1_0(\Omega)$ denote the usual space of continuous, piecewise-linear functions on~$\mathcal{T}_h$ that vanish on $\partial\Omega$. We define the space $V_{h,\tau}$ consisting of the functions~$v\in\mathcal{C}([0,T];V_h)$ whose restriction to each subinterval~$I_n=(t_{n-1},t_n)$ is a polynomial of degree at most~$1$ with coefficients in~$V_h$.

For each $\bsy\in\mathfrak{D}$, our fully-discrete solution $u_h(\cdot,\cdot,\bsy)\in V_{h,\tau}$ is then defined by requiring that
\begin{equation} \label{CNFEM}
\frac{1}{\tau_n}\int_{t_{n-1}}^{t_n}\iprod{\partial_t^\alpha u_h,v_h}\,\ud t
    +\iprod{\kappa \nabla u_h^{n-1/2},\nabla v_h}=\iprod{\bar f^n,v_h},
\end{equation}
for all $v_h\in V_h$ and for $1\le n\le\mathcal{N}_t$, where
\[u_h^n(\bsx,\bsy)=u_h(\bsx,t_n,\bsy),\quad u_h^{n-1/2}=\frac12(u_h^n+u_h^{n-1}),\quad 
\bar f^n=\frac{1}{\tau_n}\int_{t_{n-1}}^{t_n} f(t)\,\ud t. \]
Thus, $u_h(t)=\tau_n^{-1}(t-t_{n-1}) u_h^n+\tau_n^{-1}(t_n-t)u_h^{n-1}$ for~$t\in I_n$, and we impose the discrete initial condition~$u_h^0=R_h g$ where $R_h:H^1_0(\Omega) \mapsto V_h$ is the Ritz projection defined by 
$\iprod{\kappa \nabla R_h w,\nabla v_h}=\iprod{\kappa\nabla w,\nabla v_h}$ 
for all $v_h \in V_h$.  We find that
\[
\frac{1}{\tau_n}\int_{t_{n-1}}^{t_n}\iprod{\partial_t^\alpha u_h,v_h}\,\ud t
    =\sum_{j=1}^n\omega^\alpha_{nj}\iprod{u^j_h-u^{j-1}_h,v_h}
\]
where
\[
\omega_{nn}^\alpha=\frac{\omega_{3-\alpha}(\tau_n)}{\tau_n^2}
\quad\text{and}\quad
\omega_{nj}^\alpha =\frac{1}{\tau_n\,\tau_j}\int_{I_n}\int_{I_j} 
    \omega_{1-\alpha}(t-s)\,\ud s\,\ud t
    \quad\text{for $1\le j\le n-1$.}
\]

For $1\le p\le d_h:=\dim V_h$, let $\bsx_p$ denote the $p$-th interior node and let $\phi_p\in V_h$ denote the corresponding nodal basis function so that, with $u_{h,p}^n(\bsy)=u_h^n(\bsx_p,\bsy)$, we have $u^n_h(\bsx,\bsy)=\sum_{p=1}^{d_h}u_{h,p}^n(\bsy) \phi_p(\bsx)$. We define $d_h\times d_h$ matrices $\mD(\bsy)=[\iprod{\kappa(\cdot,\bsy)\nabla \phi_p,\nabla \phi_q}]$ (the  stochastic stiffness matrix) and $\mM=[\iprod{\phi_p,\phi_q}]$ (the deterministic mass matrix), together with the $d_h$-dimensional column vectors $\vUh^n(\bsy)=[u^n_{h,q}(\bsy)]$~and  $\vFh^n=[\iprod{\bar f^n,\phi_p}]$. The scheme \eqref{CNFEM} may then be written in matrix form as
\begin{equation}\label{eq: matrix form}
\mS(\bsy)\vWh^n(\bsy)=\vFh^n-\mD(\bsy)\vUh^{n-1}(\bsy)
-\sum_{j=1}^{n-1}\omega_{nj}^\alpha\mM\vWh^j(\bsy),
\quad\text{for $1\le n\le\mathcal{N}_t$,}
\end{equation}
where $\mS(\bsy)=\omega_{nn}^\alpha\mM +\frac12\mD(\bsy)$ and $\vWh^n(\bsy)= \vUh^n(\bsy)-\vUh^{n-1}(\bsy)$. Due to the non-locality in time,  storing the  increments $\vWh^j$ is required at all previous time-steps to evaluate the right-hand side of the linear system~\eqref{eq: matrix form}. The matrix~$\mS(\bsy)$ is symmetric and positive definite, and consequently, the above system is uniquely solvable at each time level~$t_n$ and for 
every~$\bsy\in\mathfrak{D}$. 

If $\alpha\to1$, then $\omega^\alpha_{nn}\to\tau^{-1}$~and $\omega^\alpha_{nj}\to0$ 
for~$0\le j\le n-1$, so \eqref{eq: matrix form} tends to the Crank--Nicolson scheme
$\bigl(\mM+\tfrac12\tau\mD(\bsy)\bigr)\vUh^n=
\bigl(\mM-\tfrac12\tau\mD(\bsy)\bigr)\vUh^{n-1}+\tau\vFh^n$ for the classical diffusion 
equation with random diffusivity \cite{NobileTempone2009}: \[\partial_tu(\bsx,t,\bsy)-\nabla\cdot(\kappa(\bsx,\bsy)\nabla u)(\bsx,t,\bsy)=f(\bsx,t).\]

For the case~$\gamma=1$ of uniform time steps of size~$\tau$, the diagonal weights
$\omega^\alpha:=\omega^\alpha_{nn}=\tau^{-2}\omega_{3-\alpha}(\tau)$ and the 
matrix~$\mS(\bsy)=\omega^\alpha\mM+\frac12\mD(\bsy)=:{\bf S}(\bsy)$ are independent of~$n$, with $\omega^\alpha=\tau^{-2}\omega_{3-\alpha}(\tau)$ and
\[
\omega_{nj}^\alpha=\omega^\alpha g_{n-j}\quad\text{for}\quad
g_j =(j+1)^{2-\alpha}-2 j^{2-\alpha}+(j-1)^{2-\alpha}.
\]
Thus, $\sum_{j=1}^{n-1}\omega_{nj}^\alpha\mM\vWh^j(\bsy)
=\omega^\alpha \sum_{j=1}^{n-1}g_j^\alpha\mathfrak{M}^{n-j}(\bsy)$ where
$\mathfrak{M}^j:=\mM \vWh^j(\bsy)$.  

In the next theorem, we state the error bound arising from the space- and time-discretizations. These bounds are $\alpha$-robust in the sense that the generic constant $C$, which is independent of $h$ and $\tau$, remains bounded as $\alpha\to 1$. We refer to \cite[Theorem 4]{MMD2024} for the proof of the first result. 
\begin{theorem}\label{Convergence theorem space-time} Let $u$ be the solution of \eqref{eq:fractional equation} and let $u_h$ be the approximate solution defined by \eqref{CNFEM}. Assume that the regularity assumptions in \eqref{time regularity} and \eqref{spatial regularity} are satisfied with $\sigma$, $\nu >\alpha/2$. Choose the mesh exponent $\gamma$ such that $\gamma>\max\{2/\sigma,2/\nu, (3-\alpha)/(2\sigma-\alpha)\}$. Then,
\[
\|u(\bsy)-u_h(\bsy)\|_{L^2(J,\Omega)}\le C(h^2+\tau^2),
\quad\text{for $\bsy\in\mathfrak{D}$}.
\]
Therefore, assuming $|\calL(w)|\le \|\calL\|\,\|w\|$ for $w\in L^2(\Omega)$, we deduce that 
\begin{equation*}
\|\calL(\vu(\bsy))-\calL(\vu_h(\bsy))\|_{L^2(J)}
\le  \|\calL\|\,\|u(\bsy)-u_h(\bsy)\|_{L^2(J,\Omega)}\le C(h^2+\tau^2)\|\calL\|
\quad\text{for $\bsy\in\mathfrak{D}$.}
\end{equation*}
\end{theorem}
\section{Estimations of the expected value}\label{Sec: main results}
This section is devoted to estimating the expected value $E(\calL(u(t)))$ in \eqref{F}. We initially approximate $u$ by~$\widehat u$, which is the solution of \eqref{eq:fractional equation} obtained by truncating the infinite expansions in\eqref{KLexpansion} or, what amounts to the same thing, setting $y_j=0$ for~$j >z$.  Then, with $\mathfrak{D}_z=[0,1]^z$ of (finite) fixed dimension $z$, we approximate 
$E(\calL(u(t)))$ by  
\begin{equation}\label{F finite}
E_z(\calL(\widehat u(t))):=\int_{\mathfrak{D}_z}
\calL\bigl(\widehat u\bigl(\cdot,t,\bsy-{\bf\tfrac{1}{2}}\bigr)\bigr)\,\ud\bsy.
\end{equation}
The shifting of the coordinates by~${\bf \frac{1}{2}}$ translates $\mathfrak{D}_z$ to $\big[-\frac{1}{2},\frac{1}{2}\big]^z$, in preparation for applying a high-order QMC rule whose points are in~$[0,1]^z$; see \eqref{eq:QMCInt F} below.

For each fixed time $t$, the error from approximating $E(\calL(u(t)))$ by $E_z(\calL(\widehat u(t)))$ may be estimated as follows, where $\mathcal{E}^\ell(t)$
was defined in~\eqref{eq: E(t)}.
\begin{theorem}\label{Truncating error}
If \eqref{error in kappa} is satisfied, then for every $\bsy\in\mathfrak{D}$ and for every $z\in\mathbb{N}$, the solution $\widehat u(\cdot,\cdot,\bsy)$ of the $z$-term truncated parametric problem \eqref{eq:fractional equation} satisfies
\begin{equation}\label{eq:Vdimtrunc}
\|u(t,\bsy)-\widehat u(t,\bsy)\|+t^{\alpha/2}\|\nabla(u(t,\bsy)-\widehat u(t,\bsy))\|
\le C z^{1-1/p} \mathcal{E}^\ell(t),
\end{equation}
for~$0\le \ell\le 1$. Moreover, since $\ud\bsy$ is the uniform probability measure 
on~$\mathfrak{D}$, 
\[
E(\calL(u(t)))-E_z(\calL(\widehat u(t)))
=\int_{\mathfrak{D}}\calL\bigl(u(t,\bsy)-\widehat u(t,\widehat\bsy)\bigr)\,\ud\bsy,
\quad\text{for $t\in(0,T]$.}
\]
Therefore, assuming $|\calL(w)|\le \|\calL\|\,\|w\|$ for $w\in L^2(\Omega)$, then, by using~\eqref{eq:Vdimtrunc}, we have 
\[ 
|\calL(u(t, \bsy)-\widehat u(t, \bsy))|\le\|\calL\|\|u(t, \bsy)-\widehat u(t, \bsy)\|
 \le C\,z^{1-1/p}\,\|\calL\|\,  \mathcal{E}^\ell(t),
\]
 for every $\bsy\in\mathfrak{D}$, and thus,
\[
|E(\calL(u(t)))-E_z(\calL(\widehat u(t)))|\le C\,z^{1-1/p} \|\calL\|\mathcal{E}^\ell(t),\quad{\rm for}~~0\le \ell\le 1\,.
\]
\end{theorem}
\begin{proof}
The function~$w(\bsx,t,\bsy)=\widehat u(\bsx,t,\bsy)-u(\bsx,t,\bsy)$ satisfies   
\[
\I^{1-\alpha} w'-\nabla\cdot(\widehat\kappa\nabla w)
    =-\nabla\cdot[(\kappa-\widehat \kappa)\nabla u]\quad\text{on $\Omega\times(0,T]$,}
\]
for every $\bsy\in\mathfrak{D}$. It is clear that $w= 0$ on $\partial\Omega\times(0,T]$ and that $w(\cdot,0,\bsy)=0$ on $\Omega$. Thus, applying \eqref{eq: bound of nabla v} and \eqref{eq: desired 4} with $\theta=(\kappa-\widehat \kappa)\nabla u$, and then using the diffusivity truncated estimate in \eqref{error in kappa}, we deduce that  
\[
\|w(t,\bsy)\|^2+t^\alpha\|\nabla w(t,\bsy)\|^2
\le C z^{1-1/p}t^{\alpha-1}\int_0^t\big(
    \|\nabla u(s,\bsy)\|^2+\|\nabla\mathfrak{U}'(s,\bsy)\|^2\Big)\,\ud s.
\]
From the regularity estimate in~\eqref{Reg x and t}, we get 
$\max\{\|\nabla u(s,\bsy)\|,\|\nabla\mathfrak{U}'(s,\bsy)\|\}
\le Cs^{-\alpha/2}\mathcal{R}^\ell(s)$ for~$0\le\ell\le 1$, and hence, 
\[
\|w(t,\bsy)\|+t^{\alpha/2}\|\nabla w(t,\bsy)\|
\le C z^{1-1/p} \mathcal{E}^\ell(t),
\quad\text{for $0\le\ell\le 1$.}
\]
Therefore, the proof of the first desired estimate is completed. 
\end{proof}

Next, we  approximate the high dimensional integral in \eqref{F finite} using a \emph{deterministic, interlaced high-order polynomial equal-weight lattice QMC rule} of the form  
\begin{equation}\label{eq:QMCInt F}
E_z(\calL(\widehat u(t)))\approx E_{z,N}(\calL(\widehat u(t)))
:=\frac{1}{N}\sum_{j=0}^{N-1}
\calL\bigl(\widehat u\bigl(\cdot,t,\bsy_j- {\bf\tfrac{1}{2}}\bigr)\bigr),
\end{equation}
for QMC points $\bsy_0$, $\bsy_1$, \dots, $\bsy_{N-1}\in\mathfrak{D}_z$. Here, $N=b^m$ for a positive integer~$m$ and a prime number $b$, and to generate the polynomial lattice rule we need a \emph{generating vector} of polynomials, $\bg=(g_1,g_2,\ldots,g_z)$ where each $g_j$ is a polynomial with degree $<m$ and coefficients taken from a finite field $\mathbb{Z}_{b}$ with $b$ elements.  In the following, let $\mathbb{Z}_{b}[x]$ be the set of all polynomials with coefficients in $\mathbb{Z}_{b}$, and let ${\mathbb Z}_{b}(x^{-1})$ be the set of all formal Laurent series $\sum_{j=q}^\infty t_j x^{-j}$, where $q$  is an arbitrary integer and the coefficients $t_j \in \mathbb{Z}_b$ for all $j.$ 

Each integer~$p$ in the range~$0\le p\le N-1$ has a unique $b$-adic representation
\[
p=(\eta_{m-1},\ldots,\eta_1,\eta_0)_b=\sum_{r=0}^{m-1}\eta_rb^r.
\]
We associate $p$ with the polynomial~$p(x)=\sum_{r=0}^{m-1}\eta_rx^r\in\mathbb{Z}_{b}[x]$, and define a map $v_{m}:{\mathbb Z}_{b}(x^{-1}) \to [0,1)$ as
\[
v_{m}\biggl(\sum_{\ell=q}^\infty t_\ell x^{-\ell}\biggr)
    =\sum_{\ell=\max(1,q)}^m t_\ell b^{-\ell},~~\text{for any integer $q$}.
\]

Let $P \in {\mathbb Z}[x]$ be an irreducible polynomial with degree $m$. The classical polynomial lattice rule $\calS_{P,b,m,z}(\bg)$ associated with $P$ and the generating vector $\bg$ comprises the quadrature points
\[
\bsy_j=\bigl(v_{m}(jg_1/P), v_m(jg_2/P), \ldots, v_{m}(jg_z/P)\bigr)\in [0,1)^z,
\quad\text{for $j = 0$, $1$, \dots, $N-1$,}
\]
where multiplication and division in $j g_i / P$ is carried out in $\mathbb{Z}(x^{-1})$.

Classical polynomial lattice rules give almost first order convergence for integrands of bounded variation. To obtain higher-order convergence, an interlacing procedure is needed. Following {\color{blue} \cite{Dick2008} } and Goda and Dick \cite{GodaDick2015}, we define $\mathscr{D}_\beta:[0,1)^\beta\to[0,1)$,  the  \emph{digit interlacing function} with digit interlacing factor~$\beta\in\mathbb{N}$, by
\begin{equation}\label{eq:DigIntl}
\mathscr{D}_\beta(x_1,\ldots,x_\beta)=\sum_{i=1}^\infty\sum_{j=1}^\beta
    \xi_{j,i}b^{-j-(i-1)\beta},
\end{equation}
where $\xi_{j,i}\in\{0,1,\ldots,b-1\}$ are the unique integers such that $x_j=\sum_{i\ge 1}\xi_{j,i}b^{-i}\in[0,1)$. In turn, the vector-valued function $\mathscr{D}_\beta^z:[0,1)^{\beta z}\to[0,1)^z$ is defined by
\[
\mathscr{D}_\beta^z(x_1,\ldots,x_{\beta z})
=\bigl(\mathscr{D}_\beta(x_1,\ldots,x_\beta), \ldots,
\mathscr{D}_\beta(x_{(z-1)\beta +1},\ldots, x_{z \beta})\bigr).
\] 
Then, an \emph{interlaced polynomial lattice rule of order $\beta$ with $b^m$ points in $z$ dimensions} is a QMC rule using $\mathscr{D}_\beta^z(\calS_{P,b,m,\beta z}(\bg))$ as quadrature points, for some given modulus $P$ and generating vector $\bg$. Next, we show the error from the high-dimensional  QMC approximation~\eqref{eq:QMCInt F}. 
\begin{theorem}\label{prop:qmc}
Assume that \eqref{ass A1}, \eqref{ass A5} and \eqref{ass A3} are satisfied, and let $b$, $m\in\mathbb{N}$ with $b$ prime. Then, one can construct an interlaced polynomial lattice rule of order $\beta:=\lfloor 1/p \rfloor+1$ with~$N=b^m$ points, using a fast component-by-component algorithm, with cost of order $\beta z N(\log N+\beta z)$  
operations, such that the following error bound holds: for $0\le \ell\le 1,$ 
\[
|E_z(\calL(\widehat u(t)))-E_{z,N}(\calL(\widehat u(t)))|\le C\,N^{-1/p}\|\calL\|\,{\mathcal E}^\ell(t),~~{\rm for}~~t \in (0,T].
\]
Here, the constant~$C$ depends on $p$~and $b$, but is independent of $z$~and $m$.
\end{theorem}
\begin{proof}
Using  \eqref{mixed estimate 2} with $\widehat u$ in place of $u$,  for each fixed $t$, we have the regularity estimate 
\[\bigl|\partial^{\bbeta}_\bsy\calL\bigl(\widehat u(t,\bsy-{\bf\tfrac{1}{2}})\bigr)\bigr|
=\bigl|\calL\bigl(\partial^{\bbeta}_\bsy\widehat u(t,\bsy-{\bf\tfrac{1}{2}})\bigr)\bigr|
\le\|\calL\|\bigl\|\partial_{\bsy}^{\bbeta}\widehat u\big(t,\bsy-{\bf\tfrac12})\bigr\|\\ 
\le C\|\calL\||\bbeta|!\,{\bf b}^{\bbeta}{\mathcal E}^\ell(t)
\]
for $0\le \ell\le 1,$ for all $\bsy\in\mathfrak{D}_z$ and for all $\bbeta\in\{0, 1,\ldots, \beta\}^z$. 
Therefore, direct application of a high-order QMC error bound of Dick et al. \cite[Theorem~3.1]{DickKuoGiaNuynsSchwab2014} yields the desired estimate. 
\end{proof}

In the next theorem, the generic constant $C$ is independent of the finite element mesh size $h$, the time mesh size $\tau$, the truncated dimension $z$, and the QMC exponent~$m$. However, $C$ may depend on $\Omega$, $T$, and other parameters including $p$, $b$, $g$, $f$, $\kappa_{\min}$, $\kappa_{\max}$, and the upper bound of $\|\nabla \kappa\|_{L^\infty(\Omega)}$.  To guarantee that $C$ remains finite as~$\alpha$ approaches~$1$, we require  that  $\mathcal{R}^\ell(t)$ in \eqref{eq: frak U cal R} to be  bounded by $c\,t^\epsilon$ for~$t>0$, for some $c$, $\epsilon>0$, and  for some $0\le \ell \le 1.$ Hence, ${\mathcal E}^\ell(t)\le c t^\epsilon /\sqrt{1-\alpha+2\epsilon}.$ This assumption, which is implicitly needed for the regularity properties in \eqref{time regularity} and \eqref{spatial regularity},    holds true under  realistic conditions on the given data $g$ and $f.$

\begin{theorem}\label{main results} 
For every $\bsy\in\mathfrak{D}$, let $u(\cdot,\cdot,\bsy)$ be the solution 
of~\eqref{eq:fractional equation} and let $\widehat u_h(\cdot,\bsy)$ be the numerical  
solution defined as in~\eqref{CNFEM} with  $\widehat\kappa$ in place 
of~$\kappa$.  Then, under the hypotheses of Theorem~\ref{Convergence theorem space-time} 
and the assumptions~\eqref{ass A1}--\eqref{ass A3}, one can construct an interlaced polynomial lattice rule of order $\beta:=\lfloor 1/p\rfloor+1$ with~$N$ points, such that
\[    
\|E(\calL(u))-E_{z,N}(\calL(\widehat u_h))\|_{L^2(J)}
    \le C\,\bigl(z^{1-1/p}+ N^{-1/p}+h^2+\tau^2\bigr)\|\calL\|.
\]
\end{theorem}
\begin{proof}
We start our proof by decomposing the error as
\begin{multline}\label{combine}
\|E(\calL(u))-E_{z,N}(\calL(\widehat u_h))\|_{L^2(J)}
    \le\|E(\calL(u))-E_z(\calL(\widehat u))\|_{L^2(J)}\\
    +\|E_z(\calL(\widehat u))-E_{z,N}(\calL(\widehat u))\|_{L^2(J)}
    +\|E_{z,N}(\calL(\widehat u-\widehat u_h))\|_{L^2(J)}.
\end{multline} 
To estimate the first and second terms in~\eqref{combine}, we apply Theorems 
\ref{Truncating error}~and \ref{prop:qmc}, respectively.  For the third term, we 
apply Theorem~\ref{Convergence theorem space-time}, setting $y_j=0$ when~$j>z$. 
\end{proof}
\section{Numerical results}\label{Sec: Numeric}

\begin{figure}
\caption{One realisation of the solution $u_h$ at time $t=0$ (left) and at $t=1$ (right).}
\label{fig: contour plots}
\vspace{1em}
\centering
\includegraphics[scale=0.57]{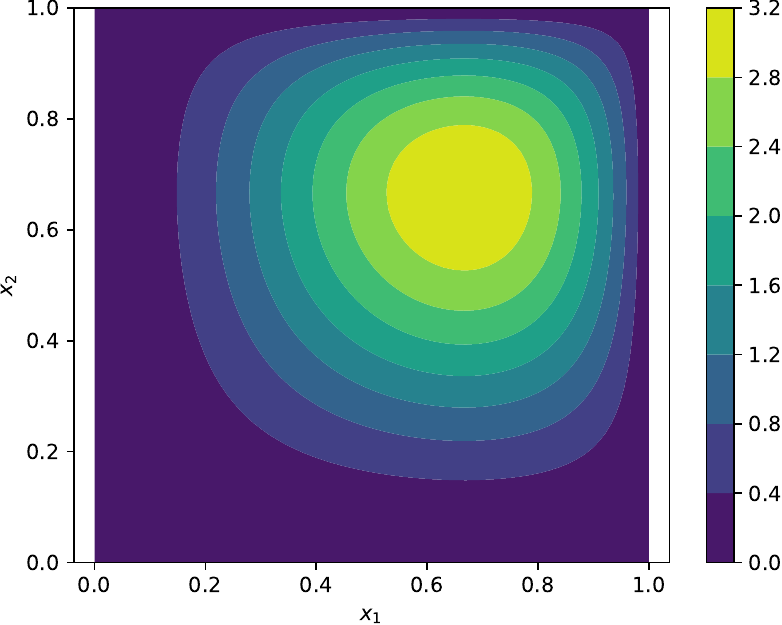}
\hfill
\includegraphics[scale=0.57]{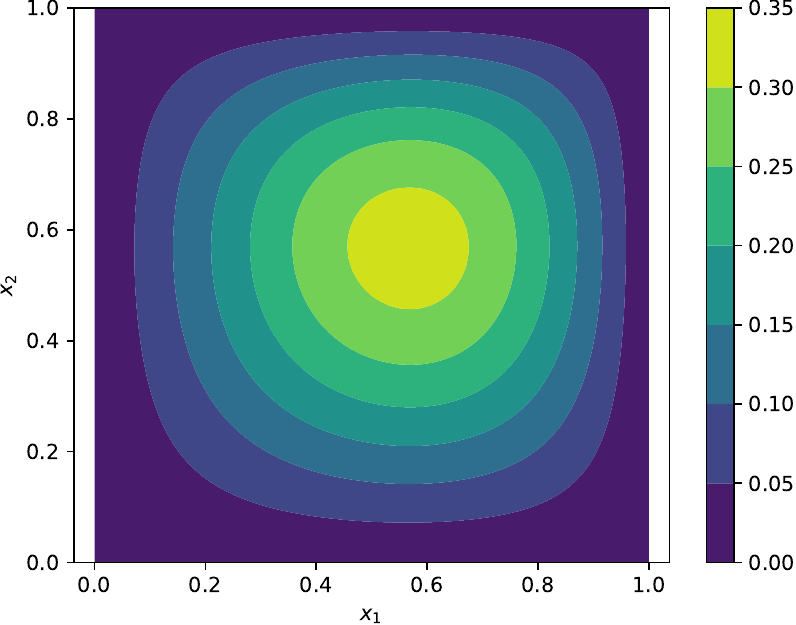}
\end{figure}

In this section, we illustrate numerically the theoretical finding in Theorem \ref{main results}. For the case of the deterministic model, we illustrated order $\tau^2+h^2$ convergence of $u_h$ in an earlier paper \cite{MMD2024}, for various choices of the initial data $g$.  In addition, on different deterministic examples, ${\mathcal O}(\tau^2)$-rates were recently illustrated in \cite{MMD2024} using the stronger $L^\infty(J,L^2(\Omega))$ norm. So, we focus in this section on confirming the  high-order QMC error estimates.  Let $\alpha=1/2$ and, for $\bsx=(x_1,x_2)$ in the unit square $\Omega=(0,1)^2$ and for $0\le t\le T=1$, define
\[
\kappa_0(\bsx)=\frac{2+x_1x_2}{10},\quad g(\bsx)=144x_1^2(1-x_1)x_2^2(1-x_2),\quad 
f(\bsx,t)=1,
\]
together with the random coefficient 
\[
\kappa(\bsx,\bsy)=\kappa_0(\bsx)+\sum_{l=1}^q\sum_{k=1}^{q+1-l}
    y_{k,l}\psi_{kl}(\bsx)\quad\text{where}\quad
    \psi_{k,l}(\bsx)=\frac{\sin(k\pi x_1)\,\sin(l\pi x_2)}{M(k+l)^4}.
\]
Here, the coefficients $y_{k,l}\in(-\tfrac12,\tfrac12)$ are ordered with the first index $k$ varying most rapidly, that is, $\bsy= (y_{1,1},y_{2,1},\ldots,y_{q,1},y_{1,2},y_{2,2},\ldots, y_{q-1,2},\ldots,y_{q,1})\in(-\frac12,\frac12)^z$ for~$z=q(q+1)/2$.  The constant $M=\zeta(3)-\zeta(4)>0$ ensures that our  assumption \eqref{eq: kappa min max} is satisfied,  with~$\zeta(s)=\sum_{n=1}^\infty n^{-s}$ denoting the Riemann zeta function, and for our computations we chose $q=22$. The number  of random coefficients~$y_{k,l}$ is therefore~$z=253$. In addition, it may be shown \cite[Section~6]{DickEtAl2024} that the assumptions \eqref{ass A4}--\eqref{ass A3} are satisfied for~$p>1/2$.  As the linear functional, we chose the average value $\calL(v)=\int_\Omega v(\bsx)\,d\bsx$ over the unit square, and note that the above normalisation of the initial data means that $\calL(g)=1$.

For the spatial discretisation, we generated a quasi-uniform triangulation~$\mathcal{T}_h$ of~$\Omega$, with maximum element diameter~$h=0.0269$ and with $2,815$ interior nodes. For the time discretisation, we defined $t_n$ as in \eqref{eq: time mesh} with $\mathcal{N}_t=150$ and $\gamma=2/\alpha=4$.  The contour plots in Figure \ref{fig: contour plots} show, for one particular choice of $\bsy$, the numerical solution at~$t=0$ and at $t=T=1$. 

For brevity, define $E_{z,N,h}(t)=E_{z,N}\bigl(\mathcal{L}(\widehat u_h(t))\bigr)$.
Table~\ref{tab: QMC convergence} lists the values of $E_{z,N,h}(T)$ produced by the 
$N$-point QMC rule for $N=16$, $32$, $64$ and $128$, along with a reference value 
computed using $N=512$.  Two sets of errors are given: those at the final time, 
$|E_{z,N,h}(T)-E_{z,512,h}(T)|$, and the $L^2$ errors 
$\|E_{z,N,h}-E_{z,512,h}\|_{L^2(J)}$.  The associated convergence rates for both are also 
shown. Since all solutions (including the  reference solution) use the same finite element 
mesh, the same time levels~$t_n$ and the same number~$z$ of random coefficients, we are 
effectively considering only the contribution to the error arising from using the QMC rule 
to approximate the expected value \eqref{F finite}. Our assumptions are satisfied for 
$p>1/2$, and the observed convergence rates are consistent with the error bound of 
order~$N^{-1/p}$ from Theorem~\ref{prop:qmc}, that is, of order $N^{-(2-\epsilon)}$ 
for any $\epsilon>0$.

\begin{table}[t]
\centering
\renewcommand{\arraystretch}{1.2}
{\tt
\begin{tabular}{c|ccc|cc}
$N$&$E_{z,N,h}(T)$&
\textrm{error at $T$}&\textrm{rate}&\textrm{error in $L^2(J)$}&\textrm{rate}\\
\hline
    16&   0.2573698441&   7.08e-05&         &   7.59e-05&         \\ 
    32&   0.2573163627&   1.73e-05&    2.031&   1.85e-05&    2.038\\ 
    64&   0.2573036000&   4.56e-06&    1.926&   4.81e-06&    1.941\\ 
   128&   0.2573001107&   1.07e-06&    2.094&   1.12e-06&    2.098\\ 
\hline
   512&   0.2572990433&
\end{tabular}
}
\caption{Convergence with respect to the number~$N$ of QMC points.}
\label{tab: QMC convergence}
\end{table}

\begin{figure}
\centering
\includegraphics[scale=0.7]{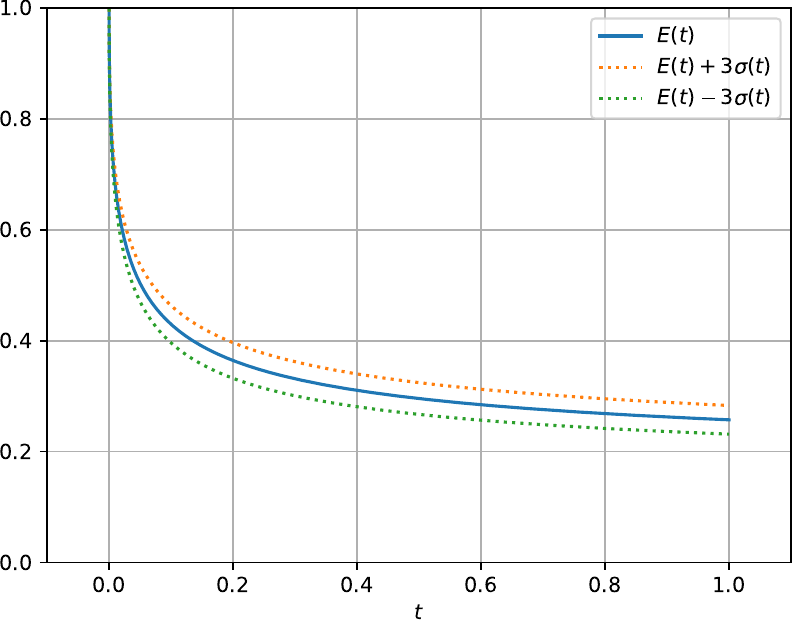}
\caption{The expected value $E(t)=E_{z,N,h}(t)$ when $N=512$.}
\label{fig: E(t)}
\end{figure}

Figure~\ref{fig: E(t)} shows how the expected value \eqref{F finite}, or rather the approximation~$E_{z,N,h}(t)$ with~$N=512$, varies with time.  Also shown are the dashed curves $E_{z,N,h}\pm3\sigma$, where $\sigma=\sigma_{z,N,h}(t)$ is the standard deviation of~$\mathcal{L}(\widehat u_h(t))$.

The Julia code used in the above computations is hosted on Github~\cite{McLean2024}, and was executed using 8~threads on a Ryzen 7 3700X CPU.  The algorithm precomputes the sparse Cholesky factorization of the matrix
\[
\mS_0(\tau)=\frac{\omega_{3-\alpha}(\tau)}{\tau^2}\,\mM+\frac12\mD({\bf 0})
\quad\text{for $\tau\in\{\,10^l:\lfloor\log_{10}(\tau_{\min})\rfloor\le l\le
\lceil\log_{10}(\tau_{\max})\rceil\,\}$},
\]
where $\tau_{\min}=\tau_1$ and $\tau_{\max}=\tau_{\mathcal{N}_t}$  and preconditions the 
matrix $\mS(\bsy)$ in the linear system \eqref{eq: matrix form} using the Cholesky factor 
with~$\tau=10^k$ for~$k=\argmin_l|\tau_n-10^l|$.  In addition, fast evaluation of the 
sum over~$j$ on the right-hand side of~\eqref{eq: matrix form} is achieved via an 
exponential sum approximation~\cite{McLean2018} of~$\omega_{1-\alpha}(t)$. For an 
overview of the software used to generate the QMC points, we refer to the survey paper 
of Kuo and Nuyens~\cite[Section 7]{KuoNuyens2016}. 

\end{document}